%% file: note_of_non-hyperbolic_measure-26-07.tex
\documentclass[11pt]{article}
\usepackage{amssymb,amsmath,amsfonts,amsthm,epsfig,latexsym,color}

\makeatletter
\newcommand{\subsectionruninhead}{\@startsection{subsection}{2}{0mm}
{-\baselineskip}{-0mm}{\bf\large}}
\newcommand{\subsubsectionruninhead}{\@startsection{subsubsection}{3}{0mm}
{-\baselineskip}{-0mm}{\bf\normalsize}}
\makeatother

\newtheorem*{theorem*}{Theorem}
\newtheorem{theoremalph}{Theorem}

\newtheorem*{maintheorem}{Main Theorem}

\newtheorem*{proposition*}{Proposition}
\newtheorem*{corollary*}{Corollary}
\newtheorem*{claim*}{Claim}
\newtheorem{theorem}{Theorem}[section]
\newtheorem{conjecture}{Conjecture}
\newtheorem{proposition}[theorem]{Proposition}

\newtheorem{lemma}[theorem]{Lemma}

\newtheorem{claim}[theorem]{Claim}
\theoremstyle{definition}
\newtheorem{definition}[theorem]{Definition}
\newtheorem{remark}[theorem]{Remark}

\numberwithin{equation}{section}

 \def\RR{{\mathbb R}}

    \def\cU{\mathcal{U}}
    \def\cV{\mathcal{V}}
\def\cE{\mathcal{E}}    
   \def\cR{\mathcal{R}}

\newcommand{\supp}{\operatorname{Supp}}
\newcommand{\diff}{\operatorname{Diff}}

\newcommand{\ind}{\operatorname{Ind}}

\begin{document}

\title{Hyperbolicity versus non-hyperbolic ergodic measures inside homoclinic classes}

\author{C. Cheng, S. Crovisier, S. Gan, X. Wang and D. Yang\footnote{
The authors would like to thank C. Bonatti, L. Wen and J. Zhang for useful comments.
S. Gan would like to thank the support of 973 project (2011CB808002) and NSFC (11231001).
X. Wang would like to thank China Scholarship Council (CSC) (201306010008) for financial support.
D. Yang would like to thank the support of NSFC (11271152) and a project funded by the Priority Academic Program Development of Jiangsu Higher Education Institutions(PAPD).}}

\maketitle

\begin{abstract}
We prove that, for $C^1$-generic diffeomorphisms, if a homoclinic class is not hyperbolic, then there is a non-hyperbolic ergodic measure supported on it.
This proves a conjecture by D\'\i az and Gorodetski~\cite{dg}.
We also discuss the conjectured existence of periodic points with different stable dimension in the class.
\end{abstract}

\section{Introduction}\label{Section:introduction}
\subsection{Backgrounds and main results}
It is a major problem for dynamists to study dynamics beyond uniform hyperbolicity since 1960s when Abraham and Smale~\cite{as} found that
the set of \emph{hyperbolic diffeomorphisms} (i.e. satisfying the axiom A and the no-cycle condition) is not dense in dynamical systems. The lack of hyperbolicity may be characterized in different ways.
Let $M$ be a compact smooth manifold without boundary of dimension $d$ and denote by $\diff^r(M)$ the space of $C^r$-diffeomorphisms of $M$ for any $r\geq 1$.
\begin{itemize}
\item[--] Thanks to the contributions~\cite{aoki,hayashi,liao,mane} to the $C^1$-stability conjecture,
it is known that any non-hyperbolic diffeomorphism can be perturbed as a diffeomorphism with a non-hyperbolic periodic orbit.
In particular a $C^1$-generic diffeomorphism which is not hyperbolic has arbitrarily weak periodic orbits.
\item[--] Palis has proposed other obstructions to hyperbolicity, still related to periodic orbits:
he conjectured~\cite{pa} that any diffeomorphism can be $C^r$-approximated by one which is hyperbolic or by one which exhibits
a homoclinic bifurcation. These bifurcations (homoclinic tangencies and heterodimensional cycles) have strong dynamical consequences
and can sometimes be strengthened as robust obstructions to hyperbolicity, see~\cite{bd2}. Several progresses have been obtained in the direction
of this conjecture~\cite{c,cp,csy,ps}.
\item[--] Pesin's theory~\cite{pesin} weakens the notion of hyperbolicity (non-uniform hyperbolicity), and gives a possible approach to characterize non-hyperbolic behavior
through the invariant measures. This is the goal of the present paper.
\end{itemize}
Recall that for an ergodic measure $\mu$ of a diffeomorphism $f\in\diff^1(M)$, there are $d$ numbers $\chi_1\leq\chi_2\leq\cdots\leq\chi_d$, such that, for $\mu$-a.e. point $x\in M$, and for any $v\in T_xM\setminus\{0\}$, we have $\lim_{n\rightarrow+\infty}\frac{1}{n}\log\|Df^n(v)\|=\chi_i$ for some $i\in\{1,2,\cdots, d\}$. The $d$ numbers $\chi_i$ are called the \emph{Lyapunov exponents} of the measure $\mu$. We call $\mu$ a \emph{hyperbolic measure}, if all its Lyapunov exponents are non-zero.

Clearly if a diffeomorphism has a non-hyperbolic ergodic measure it can not be hyperbolic. The converse is not true in general, see~\cite{bbs,clr}, but we may expect that it is the case
for typical systems, as conjectured by D\'\i az and Gorodetski in~\cite{dg}.

\begin{conjecture}[D\'\i az, Gorodetski]\label{Conj:Palis in measure version}
There is an open dense subset $\mathcal{U}\subset\diff^r(M)$ where $r\geq 1$, such that, every diffeomorphism $f\in\mathcal{U}$ either is uniformly hyperbolic or has an ergodic non-hyperbolic invariant measure.
\end{conjecture}
\bigskip

In order to study the dynamics, one usually concentrates on the set of points that have some recurrence properties,
for instance on the \emph{chain-recurrent set} which splits into disjoint invariant compact sets called \emph{chain-recurrence classes}.
For $C^1$-generic diffeomorphisms (i.e. diffeomorphisms in a dense G$_\delta$ subset of $\diff^1(M)$), periodic points are dense
in the chain-recurrent set and any chain-recurrence class which contains a (hyperbolic) periodic point $p$, coincides with its
\emph{homoclinic class} $H(p)$, that is the closure of the set of traverse intersection points between the stable and unstable
manifolds of $orb(p)$. See~\cite{bc,c-chain-transitive}. There may exist other chain-recurrence classes called \emph{aperiodic classes}.

This viewpoint allows to derive local versions of the previous problems:
In~\cite{w}, it is shown that any non-hyperbolic homoclinic class of a $C^1$-generic diffeomorphism
contains arbitrarily weak periodic orbits.
Versions of Palis conjecture inside homoclinic classes have been partially addressed in the works mentioned above, in particular~\cite{cp}.
A local version of the conjecture above for homoclinic classes has been also stated in~\cite{dg}.

\begin{conjecture}[D\'\i az, Gorodetski]\label{Conj:non-hyperbolic measure}
For generic $f\in\diff^1(M)$, every homoclinic class either is uniformly hyperbolic or supports an ergodic non-hyperbolic invariant measure.
\end{conjecture}

We point out here that, Conjecture~\ref{Conj:non-hyperbolic measure} is not true without the genericity assumption. \cite{clr,r} construct a homoclinic class containing a homoclinic tangency inside by destroying hyperbolic horseshoes in a parameterized families of diffeomorphisms on surface. Such a homoclinic class is not hyperbolic because of the existence of homoclinic tangencies, but it is uniformly hyperbolic in the measure sense. To be precise, all Lyapunov exponents of all invariant measures supported on the homoclinic class are uniformly bounded away from 0. However, a homoclinic tangency is not persistent under perturbations, hence one expects a positive answer to Conjecture~\ref{Conj:non-hyperbolic measure}, which is what we prove in this paper.

\begin{maintheorem}
For generic $f\in\diff^1(M)$, if a homoclinic class $H(p)$ is not hyperbolic, then there is a non-hyperbolic ergodic measure $\mu$, such that $\supp(\mu)\subset H(p)$.
\end{maintheorem}

Let us emphasize that the same result does not hold for aperiodic classes:
\cite{bcs} builds an open set of diffeomorphisms whose $C^1$-generic elements have
aperiodic classes which only support hyperbolic ergodic measures.
We do not have an answer to Conjecture~\ref{Conj:Palis in measure version}. A possible approach would be to answer the following problem,
see~\cite[Conjecture 1]{abcd}, \cite[Conjecture 1]{b} or \cite[Section 6.2.Ia]{c-survey-icm}.

\begin{conjecture} For any $C^1$-generic diffeomorphism which is not hyperbolic,
there exists a non-hyperbolic homoclinic class.
\end{conjecture}
\bigskip

Let us comment the technics for getting ergodic measures with one zero Lyapunov exponent.

In the special case a homoclinic class $H(p)$ has a partially hyperbolic splitting $E^s\oplus E^c\oplus E^u$, $\dim(E^c)=1$,
then the class is hyperbolic if and only if all its ergodic measures have a non-zero central Lyapunov exponent, with the same sign;
any hyperbolic measure is approximated by hyperbolic periodic orbits in the class with the same stable dimension (see~\cite{c} and Proposition~\ref{Pro:dominated-measure} below).
Hence if $H(p)$ is not hyperbolic, either it contains a non-hyperbolic ergodic measure (as required),
or it contains two hyperbolic periodic orbits with different stable dimensions.
In this second case, it is possible to build a non-hyperbolic measure by mixing the two period orbits,
using certain shadowing properties. This has been developed by many works~\cite{gikn, dg, bdg, bbd} among others.
A partial hyperbolicity on a subset of the class is enough for this argument,
providing the following partial answer to Conjecture~\ref{Conj:non-hyperbolic measure}.

\begin{theorem*}[D\'\i az-Gorodetski]\label{Cor:non-hyperbolic measure with full support}
For generic $f\in\diff^1(M)$, if a homoclinic class contains periodic points with different stable dimensions, then there is a non-hyperbolic ergodic measure $\mu$ supported on it.
\end{theorem*}

The main difficulty in this paper is to obtain periodic points with different stable dimensions in a same
non-hyperbolic homoclinic class. This is achieved under new settings, thanks to the recent result~\cite{w}
(which produces weak periodic orbits)
combined to~\cite{bcdg} (in order to turn them to heterodimensional cycles).
There are still cases where we do not manage to get such periodic points
(see the discussion in Section~\ref{ss.cycle} below), but the existence of non-hyperbolic measures is ensured then.

In the following, we call the stable dimension of a hyperbolic periodic point $p$ the \emph{index} of $p$, and denote it by $\ind(p)$. Recall that for a positive integer $T$, a \emph{$T$-dominated splitting} over an invariant compact set $K$ is a continuous $Df$-invariant splitting $T_KM=E\oplus F$, such that, $\|Df^T|_{E(x)}\|\cdot\|Df^{-T}|_{F(f^T(x))}\|<\frac{1}{2}$, for any $x\in K$. We say $K$ has a \emph{dominated splitting}, if $K$ has a $T$-dominated splitting for some $T$. Moreover, the dimension of the bundle $E$ is called the \emph{index} of the dominated splitting.
\medskip

The main theorem follows immediately from the following two theorems, considering whether the homoclinic class admits a dominated splitting corresponding to the index of $p$ or not. The proofs of Theorem~\ref{Thm:dominated} and Theorem~\ref{Thm:non-dominated} are given in Section~\ref{Section:domination} and Section~\ref{Section:non-domination} respectively.

\begin{theoremalph}\label{Thm:dominated}

For generic $f\in\diff^1(M)$, if the homoclinic class $H(p)$ of a periodic point $p$ of index $i$ admits a dominated splitting $E\oplus F$ with $\dim E=i$, and if the bundle $E$ is not uniformly contracted, then there exists an ergodic measure supported on $H(p)$ whose $i^\text{th}$ Lyapunov exponent equals $0$.

\end{theoremalph}

\begin{theoremalph}\label{Thm:non-dominated}

For generic $f\in\diff^1(M)$, if the homoclinic class $H(p)$ of a periodic point $p$ of index $i$ does not admit a dominated splitting $E\oplus F$ with $\dim E=i$, then there exists a non-hyperbolic ergodic measure $\mu$ such that ${\rm supp}(\mu)=H(p)$.
Moreover, if the $i^\text{th}$ and $({i+1})^\text{th}$ Lyapunov exponents $\chi_i$, $\chi_{i+1}$ of $p$
satisfy $\chi_i+\chi_{i+1}<0$, then the $({i+1})^\text{th}$ Lyapunov exponent of $\mu$ vanishes.

\end{theoremalph}

One can ask if the support of the non-hyperbolic measure can coincide with the whole homoclinic class.
We got a partial answer which generalizes~\cite{bdg}.
The proof is given in Section~\ref{Section:Proof of consequences}.

\begin{proposition}\label{Pro:non-hyperbolic measure with full support}
For generic $f\in\diff^1(M)$, if a homoclinic class $H(p)$ contains periodic points with different indices, then there is a non-hyperbolic ergodic measure $\mu$ with ${\rm supp}(\mu)=H(p)$.
\end{proposition}

Further questions about the non-hyperbolic ergodic measures supported on the class
may be asked. C. Bonatti proposed the following two:
\smallskip

\noindent
$(1)$ \emph{Is there such a measure with positive entropy?}\\
This has been obtained in~\cite{bbd} for homoclinic classes containing periodic points with different indices.
\smallskip

\noindent
$(2)$ \emph{Consider the set of ergodic measures whose ${i}^\text{th}$ Lyapunov exponent vanishes. Is it a convex set?}

\subsection{Heterodimensional cycles inside non-hyperbolic homoclinic classes}\label{ss.cycle}

Bonatti-D\'\i az \cite{bd} proposed a generalized Palis' conjecture for the $C^1$-topology:
heterodimensional cycles (i.e. the existence of two hyperbolic periodic orbits with different indices
linked by heteroclinic orbits) should appear densely in the interior of the set of non-hyperbolic $C^1$-diffeomorphisms.
Here we state a local version of this conjecture for homoclinic classes.

\begin{conjecture}[\cite{b,bcdg,bd2,c-survey-icm}]\label{Conj:cycle versus hyperbolicty}
For any generic $f\in\diff^1(M)$, if a homoclinic class $H(p)$ is not hyperbolic, then arbitrarily $C^1$-close to $f$, there is a diffeomorphism $g$ that exhibits a heterodimensional cycle associated to $p_g$, where $p_g$ is the continuation of $p$.
\end{conjecture}

Note that since on surfaces there is no heterodimensional cycle, these conjectures imply ``Smale's'' conjecture:
hyperbolic systems are dense in the space of $C^1$ surface diffeomorphisms.
\medskip

This conjecture, if satisfied, would generalize the Proposition~\ref{Pro:non-hyperbolic measure with full support} above to any non-hyperbolic homoclinic class.
For this reason we discuss partial results known in this direction.
The next two statements are consequences of the results in~\cite{bcdg} and~\cite{w}
and are proved in Section~\ref{Section:Proof of consequences}.

\begin{proposition}\label{Pro:partial hyperbolicity versus cycle}
For any generic $f\in\diff^1(M)$, and for any non-hyperbolic homoclinic class $H(p)$ associated to a hyperbolic saddle $p$ of index $i$, with $2\leq i\leq d-2$,  one of the following two possibilities holds:
\begin{itemize}
\item[--] $H(p)$ contains a periodic point with different index,
\item[--] $H(p)$ has a partially hyperbolic splitting $E^s\oplus E^c$, $\dim(E^s)=i-1$,
or $E^c\oplus E^u$, $\dim(E^c)=i+1$.
\end{itemize}
\end{proposition}

\begin{proposition}\label{Pro:partial hyperbolicity versus cycle of index 1}
Assume $\dim(M)=d\geq 3$.
For any generic $f\in\diff^1(M)$, and for any non-hyperbolic homoclinic class $H(p)$ associated to a hyperbolic saddle $p$ of index $d-1$, one of the following three possibilities holds:
\begin{itemize}
\item[--] $H(p)$ contains a periodic point with different index,
\item[--] $H(p)$ has a partially hyperbolic splitting $E^s\oplus E^c$, $\dim(E^s)=d-2$,
\item[--] $H(p)$ is the Hausdorff limit of periodic sinks.
\end{itemize}
\end{proposition}

This last property has to be compared to a similar statement on surfaces:

\begin{theorem*}[Pujals-Sambarino~\cite{ps}]
For a generic surface diffeomorphism, any non-hyperbolic homoclinic class is the Hausdorff
limit of periodic sinks or sources.
\end{theorem*}
\bigskip

Based on these propositions and on the previous known results~\cite{bbd,bc-central-manifolds,bcdg,bd2,c,cps,w},
one can list the different possibilities of a non-hyperbolic homoclinic class
and discuss Conjecture~\ref{Conj:cycle versus hyperbolicty} in each case.
\medskip

\subsubsection*{Non-hyperbolic homoclinic classes $H(p)$ for $C^1$-generic diffeomorphisms}

\paragraph{Case a --} \emph{There exist two periodic points of different index.}\\
This is the case satisfied on examples
and which corresponds to the Conjecture~\ref{Conj:cycle versus hyperbolicty}.

\paragraph{Case b --} \emph{All periodic points have the same index $i$ and the class
$H(p)$ has a dominated splitting $T_{H(p)}M=E^s\oplus E^c_1\oplus E^c_2\oplus E^u$, with $\dim(E^c_j)\in \{0,1\}$,
and $i=\dim(E^s\oplus E^c_1)$. Maybe $E^s$ and/or $E^u$ is trivial.}\\
This is exactly the case which occurs~\cite{c} when $f$ is far from homoclinic tangencies and heterodimensional cycles, hence it is the case in the spirit of Palis conjecture.

\paragraph{Case c --} \emph{All periodic points have the same index $i$,
the class $H(p)$ has a dominated splitting $T_{H(p)}M=E^s\oplus E^c\oplus E^u$, with $\dim(E^c)=2$,
and $\dim(E^s)=i-1$, the bundle $E^c$ does not split, and there exist periodic points
which contract and others which expand the volume along $E^c$. Maybe $E^s$ and/or $E^u$ is trivial.}\\
Periodic points which expand the volume along $E^c$ are dense in the class, it is thus possible to turn them
into points of index $i-1$. Since the class has only points of index $i$ (even after perturbation),
the strong stable manifold of such a periodic point has to intersect $H(p)$
only at the periodic point itself. One may then expect that for any point $x$ in the class $W^{ss}(x)\cap H(p)=\{x\}$.
In this case by~\cite{bc-central-manifolds} the class is contained in a submanifold tangent to $E^c\oplus E^u$.
Arguing in a same way with periodic points which contract the volume along $E^c$,
one deduces that $H(p)$ is contained in a locally invariant surface tangent to $E^c$.
We are thus reduced to Smale's conjecture.

\paragraph{Case d --} \emph{All periodic points have the same index $i$,
the class $H(p)$ has a dominated splitting $T_{H(p)}M=E\oplus E^u$, with $\dim(E)=i+1$,
there is no dominated splitting corresponding to index $i$ and along any periodic orbit
the volume of planes in $E$ is contracted (sectional dissipation in $E$).}\\
As in case (c), one can expect that the class is contained in a locally invariant
submanifold tangent to $E$. We are thus reduced to the case of
a homoclinic class whose periodic points have one-dimensional unstable spaces
and sectional dissipative and has no domination corresponding to index $\dim(M)-1$.
We are thus reduced to a generalized Smale's conjecture for higher dimension,
as described in~\cite[Conjecture 8]{b}.

\paragraph{Case d' --} \emph{Similar to case (d) but for $f^{-1}$.}\\
Here again, one may expect to reduce to the generalized Smale's conjecture.
\bigskip

In Section~\ref{Section:Proof of consequences} we prove:

\begin{proposition}\label{Pro:classify homoclinic classes}
For a generic diffeomorphism in $\diff^1(M)$, any non-hyperbolic homoclinic class
has to satisfy one of the cases above.
\end{proposition}

\input{preliminary.tex}

\input{dominated.tex}

\input{non-dominated.tex}

\input{proof-of-consequences.tex}


\flushleft{\bf Cheng Cheng} \\
School of Mathematical Sciences, Peking University, Beijing, 100871, P.R. China\\
\textit{E-mail:} \texttt{chocolate-74@163.com}\\

\flushleft{\bf Sylvain Crovisier} \\
CNRS - Laboratoire de Math\'ematiques d'Orsay, Universit\'e Paris-Sud 11, Orsay 91405, France\\
\textit{E-mail:} \texttt{Sylvain.Crovisier@math.u-psud.fr}\\

\flushleft{\bf Shaobo Gan} \\
School of Mathematical Sciences, Peking University, Beijing, 100871, P.R. China\\
\textit{E-mail:} \texttt{gansb@pku.edu.cn }\\

\flushleft{\bf Xiaodong Wang} \\
School of Mathematical Sciences, Peking University, Beijing, 100871, P.R. China\\
Laboratoire de Math\'ematiques d'Orsay, Universit\'e Paris-Sud 11, Orsay 91405, France\\
\textit{E-mail:} \texttt{xdwang1987@gmail.com}\\

\flushleft{\bf Dawei Yang} \\
School of Mathematical Sciences, Soochow University, Suzhou, 215006, P.R. China\\
\textit{E-mail:} \texttt{yangdw1981@gmail.com, yangdw@suda.edu.cn}\\

\end{document}

%% file: preliminary.tex
\section{Preliminary}\label{Section:preliminary}

\subsection{Hyperbolicity and Lyapunov exponents}

For an $f$-invariant measure $\mu$, we list all its Lyapunov exponents as $\chi_1(\mu,f)\le\chi_2(\mu,f)\le\cdots\le\chi_d(\mu,f)$. Denote by $\chi_i(\mu)$ if there is no ambiguity. We define a function
   \begin{displaymath}
     L_i(\mu,f)=\liminf_{m\rightarrow+\infty} \int L_i^{(m)}(x,f)d\mu(x),
   \end{displaymath}
where
   \begin{displaymath}
     L_i^{(m)}(x,f)=\frac{1}{m}\log \|\wedge^iDf^m(x)\|.
   \end{displaymath}
Then $\chi_i(\mu,f)=L_{d-i+1}(\mu,f)-L_{d-i}(\mu,f)$. In particular, if $\mu$ is ergodic, then for $\mu$-a.e. $x\in M$, we have
   \begin{displaymath}
     L_i(\mu,f)=\lim_{m\rightarrow+\infty} L_i^{(m)}(x,f),
   \end{displaymath}
and
   \begin{displaymath}
     \chi_i(\mu,f)=\lim_{m\rightarrow+\infty} (L_{d-i+1}^{(m)}(x,f)-L_{d-i}^{(m)}(x,f)).
   \end{displaymath}

Now we recall the definition of hyperbolicity of an invariant compact set, and give some notations.

\begin{definition}\label{Def:hyperbolicity}
Assume $f\in\diff^1(M)$. An invariant compact set $K$ is \emph{hyperbolic}, if there is a continuous $Df$-invariant splitting $T_KM=E^s\oplus E^u$, such that $E^s$ is contracted by $Df$ and $E^u$ is expanded by $Df$. To be precise, there are two constants $C>0$ and $\lambda\in(0,1)$, such that, for any point $x\in K$, and any $n\geq 0$, the following two inequalities are satisfied.
   \begin{displaymath}
     \|Df^n|_{E(x)}\|<C\lambda^n,\text{ and } \|Df^{-n}|_{F(x)}\|<C\lambda^n.
   \end{displaymath}
In particular, a periodic point $p$ is a \emph{hyperbolic periodic point} if its orbit is a hyperbolic set, and the dimension of $E^s$ is called the \emph{index} of $p$, denoted by $\ind(p)$. A hyperbolic saddle $p$ of index $i$ is said to be \emph{center-dissipative} if $\chi_i(p,f)+\chi_{i+1}(p,f)<0$.
\end{definition}

\begin{remark}\label{Rem:continuation of hyperbolic periodic point}
It is well known that a hyperbolic periodic point always has a \emph{continuation}. More precisely, for a hyperbolic periodic point $p$ of a diffeomorphism $f\in\diff^1(M)$, there is a $C^1$-neighborhood $\cU$ of $f$, and a neighborhood $U$ of $orb(p)$, such that any diffeomorphism $g\in\cU$ has a unique periodic orbit contained in $U$. Moreover, this periodic orbit is hyperbolic and has the same period as $p$. We denote by $p_g$ the continuation of $p$ for any $g\in\cU$.
\end{remark}

Sometimes, a non-hyperbolic set has some weaker hyperbolicity.


\begin{definition}\label{Def:partial hyperbolicity}
Let $f\in\diff^1(M)$. An invariant compact set $K$ is \emph{partially hyperbolic}, if there is a continuous splitting $T_KM=E^s\oplus E^c\oplus E^u$, such that $(E^s\oplus E^c)\oplus E^u$ and $E^s\oplus (E^c\oplus E^u)$ are dominated, the bundle $E^s$ is contracted by $Df$, the bundle $E^u$ is expanded by $Df$ and at least one of $E^s$ and $E^u$ is non-trivial.
\end{definition}



\subsection{Sufficient conditions for existence of a non-hyperbolic ergodic measure}

We define a relationship called \emph{multiple almost shadowing} between two periodic orbits, which was called good approximation in~\cite{bdg,dg}.

\begin{definition}[Multiple almost shadowing]\label{Def:multiple almost shadow}
Consider two periodic points $p$ and $q$ of a map $f:M\rightarrow M$ and two numbers $\gamma>0$ and $0<\varkappa\leq 1$. Denote by $\pi(p)$ the period of $p$. We say $orb(p)$ is $(\gamma,\varkappa)$-\emph{multiple almost shadowed} by $orb(q)$, if there are a subset $\Gamma\subset orb(q)$ and a map $\rho:\Gamma\rightarrow orb(p)$, that satisfy the following properties.
\begin{itemize}
\item[--] $\frac{\#\Gamma}{\#orb(q)}\geq\varkappa$.
\item[--] $\#\rho^{-1}(f^j(p))$ is constant.
\item[--] $d(f^j(p),f^j(\rho(p)))<\gamma,$ for any $j=0,1,\cdots,\pi(p)-1$.
\end{itemize}
\end{definition}

One says a periodic orbit $p$ has \emph{simple spectrum}, if the $d$ Lyapunov exponents of $orb(p)$ are mutually different. The following lemma is standard, see a similar statement in~\cite[Theorem 3.5]{bdg}.

\begin{lemma}\label{Lem:m.a.s.}
For generic $f\in\diff^1(M)$, consider a periodic point $p$ with simple spectrum whose homoclinic class $H(p)$ is non-trivial. Then for any $\varepsilon,\gamma>0$, and any $\varkappa\in(0,1)$, there is a periodic point $q$ with simple spectrum homoclinically related to $p$, such that the following properties are satisfied:
\begin{enumerate}
\item\label{item:dense orbit} $orb(q)$ is $\varepsilon$-dense in $H(p)$,
\item\label{item:m.a.s.} $orb(p)$ is $(\gamma,\varkappa)$-multiple almost shadowed by $orb(q)$.
\item\label{item:exponent close} $\chi_i(q,f)$ is $\varepsilon$-close to $\chi_i(p,f)$, for any $i=1,2,\cdots,d$.
\end{enumerate}

In particular, if $p$ is center-dissipative, then $q$ can be chosen to be center-dissipative.
\end{lemma}

\begin{proof}[Sketch of the proof]
The properties in the statement are persistent under $C^1$ perturbations. Hence we only show that one can obtain such a periodic point $q$ by $C^1$-small perturbations, and by a standard Baire argument, the statement holds for generic systems.

Since $p$ has simple spectrum, there is a dominated splitting $E_1\oplus E_2\oplus\cdots\oplus E_d$ over $orb(p)$, such that, each $E_i$ is the one dimensional sub-bundle corresponding to the $i^{th}$ Lyapunov exponent of $orb(p)$. Take a transverse homoclinic point $x\in W^s(p)\cap W^u(p)$, such that $orb(x)\cup orb(p)$ is $\frac{\varepsilon}{2}$-dense in $H(p)$. By similar technics to the proof of~\cite[Proposition 1.10, Page 689]{c}, arbitrarily $C^1$-close to $f$ in $\diff^1(M)$, there is a diffeomorphism $g$, which coincides with $f$ on $orb(x)$ and out side a small neighborhood of the point $x$, such that the dominated splitting $E_1\oplus E_2\oplus\cdots\oplus E_d$ can be spread to the set $orb(x)\cup orb(p)$. Since $orb(x)\cup orb(p)$ is still a hyperbolic set with respect to $g$, by the shadowing lemma, there is a hyperbolic periodic point $q$ homoclinically related to $p$, such that $orb(q)$ spends an arbitrarily large portion of time close to $orb(p)$, and $\frac{\varepsilon}{2}$-shadows $orb(x)\cup orb(p)$. Then the items~\ref{item:dense orbit},~\ref{item:m.a.s.} are satisfied.
Since the dominated splitting $E_1\oplus E_2\oplus\cdots\oplus E_d$ with each bundle of dimension one spreads to the set $orb(x)\cup orb(q)$, each Lyapunov exponent $\chi_i(q,g)$ can be presented as $\frac{1}{\pi(q,g)}\log\|Dg^{\pi(q,g)}|_{E_i(q)}\|$. Then the item~\ref{item:exponent close} can be obtained by the fact that $orb(q)$ spends most of the time close to $orb(p)$.
\end{proof}

The following lemma is proved in~\cite{bdg} to obtain ergodic measures, using the method developed in~\cite{gikn}, by taking the weak-$\ast$-limit of atomic measures supported on periodic orbits.

\begin{lemma}[Lemma 2.5 of~\cite{bdg}]\label{Lem:conditions for ergodicity}
Consider two sequences of numbers $(\gamma_n)_{n\geq 1}$ in $(0,+\infty)$, and $(\varkappa_n)_{n\geq 1}$ in $(0,1]$, such that $\sum_{n=1}^{\infty}\gamma_n<\infty$ and $\prod_{n=1}^{\infty}\varkappa_n>0$. Assume $(p_n)_{n\geq 1}$ is a sequence of periodic points of a map $f:M\rightarrow M$ with increasing periods $\pi(p_n)$, and denote by $\mu_n$ the probability atomic measure uniformly distributed on the orbit of $p_n$. If $orb(p_{n})$ is $(\gamma_n,\varkappa_n)$-multiple almost shadowed by $orb(p_{n+1})$ for any $n\geq 1$, then the sequence of measures $(\mu_n)$ converges to an ergodic measure $\mu$ and $\supp(\mu)=\bigcap_{n=1}^{\infty}(\overline{\bigcup_{k=n}^{\infty}orb(p_k)})$.
\end{lemma}

\subsection{Perturbation lemmas about periodic cocycles}

Consider a family of linear maps $A_1,\cdots,A_n\in GL(d,\RR)$. Denote by $B=A_n\circ\cdots\circ A_1$ and denote by $\lambda_1(B),\cdots,\lambda_d(B)$ the eigenvalues of $B$ counted by multiplicity such that $|\lambda_1|\leq\cdots\leq|\lambda_d|$. Then the $i^{th}$ \emph{Lyapunov exponent} of $B$ is defined as $\chi_i(B)=\frac{1}{n}\log|\lambda_i(B)|$.

The following statement follows from~\cite{bb}, which allows to modify only two consecutive Lyapunov exponents of a cocycle, see also~\cite[Lemma 4.4]{bcdg}. A similar result can also be found in~\cite[Theorem 1.2]{g}.

\begin{lemma}[\cite{bb},Theorem 4.1 and Proposition 3.1]\label{Lem:conditions for domination of given index}
For any constants $D>1$, $\varepsilon>0$ and $d\geq 2$, there are two constants $T$ and $n_0$ satisfying the following property.

Consider a family of linear maps $A_1,\cdots,A_n\in GL(d,\RR)$ with $n\geq n_0$, such that $\|A_i\|,\|A_i^{-1}\|\leq D$, and the linear map $B=A_n\circ\cdots\circ A_1$ has no dominated splitting of index $i$ for some $i\in\{1,\cdots,d-1\}$. Then there exist one-parameter families of linear maps $(A_{m,t})_{t\in[0,1]}$ for any $1\leq m\leq n$, satisfying the following properties.
\begin{enumerate}
\item $A_{m,0}=A_m$, for any $1\leq m\leq n$.
\item $\|A_{m,t}-A_m\|<\varepsilon$ and $\|A^{-1}_{m,t}-A^{-1}_m\|<\varepsilon$, for any $t\in[0,1]$ and for any $1\leq m\leq n$.
\item Consider the linear map $B_t=A_{n,t}\circ\cdots\circ A_{1,t}$, then the Lyapunov exponents of $B_t$ satisfy the following properties.
    \begin{itemize}
    \item[--] $\chi_j(B_t)=\chi_j(B)$ for any $j\neq i,i+1$,
    \item[--] $\chi_i(B_t)+\chi_{i+1}(B_t)=\chi_i(B)+\chi_{i+1}(B)$,
    \item[--] $\chi_i(B_t)$ is non-decreasing and $\chi_{i+1}(B_t)$ is non-increasing, that is
     \begin{center}
      $\chi_i(B_t)\leq\chi_i(B_{t'})\leq \chi_{i+1}(B_{t'})\leq \chi_{i+1}(B_{t})$, for any $t<t'$,
     \end{center}
    \item[--] $\chi_i(B_1)=\chi_{i+1}(B_1)$.
    \end{itemize}
\end{enumerate}

\end{lemma}

We state a generalized version of Franks' Lemma ~\cite{f}, which is proved in~\cite{gourmelon}.

\begin{lemma}[\cite{gourmelon}]\label{Lem:franks-gourmelon lemma}
Consider a diffeomorphism $f\in\diff^1(M)$, a hyperbolic periodic point $q$ of period $\pi$ and a constant $\varepsilon>0$. Assume that for any $n=0,1,\cdots,\pi-1$, there is a one-parameter family of linear maps $(A_{n,t})_{t\in[0,1]}$ in $GL(d,\RR)$, the following properties are satisfied:
\begin{itemize}
\item[--] $A_{n,0}=Df(f^n(q))$,
\item[--] $\|Df(f^n(q))-A_{n,t}\|<\varepsilon$ and $\|Df^{-1}(f^n(q))-A^{-1}_{n,t}\|<\varepsilon$, such that, for any $t\in[0,1]$,
\item[--] $A_{\pi-1,t}\circ\cdots\circ A_{0,t}$ is hyperbolic for any $t\in[0,1]$,.
\end{itemize}
Then, for any neighborhood $V$ of $orb(q)$, any constant $\delta>0$, and any pair of compact sets $K^s\subset W^s_{\delta}(q,f)$ and $K^u\subset W^u_{\delta}(q,f)$ that are disjoint from $V$, there is a diffeomorphism $g$ that is $\varepsilon$-close to $f$ in $\diff^1(M)$, and that satisfies the following properties:
\begin{enumerate}
\item $g$ coincides with $f$ on $M\setminus V$ and $orb(q)$;
\item $K^s\subset W^s_{\delta}(q,g)$ and $K^u\subset W^u_{\delta}(q,g)$;
\item $Dg(g^n(q))=Dg(f^n(q))=A_{n,1}$ for all $n=0,\cdots,\pi-1$.
\end{enumerate}
\end{lemma}

\subsection{Generic properties}
Recall that a subset $\cR$ of a topological Baire space $X$ is called a \emph{residual} set, if it contains a dense $G_{\delta}$ set of $X$. A property is a \emph{generic} property of $X$, if there is a residual set $\cR\subset X$, such that each element contained in $\cR$ satisfies the property.
The following theorem summarizes some classical generic properties, see for instance~\cite{abcdw,bc,bd,bdv,bgw,dg}.

\begin{theorem}\label{Thm:generic properties}
There is a residual subset $\cR\subset\diff^1(M)$, such that, for any diffeomorphism $f\in\cR$, the following properties are satisfied.
\begin{enumerate}
\item\label{Generic:kupka-smale} The diffeomorphism $f$ is Kupka-Smale: all periodic points are hyperbolic and the stable and unstable manifolds of any two periodic orbits intersect transversely. Moreover, every periodic point is center-dissipative with respect to $f$ or to $f^{-1}$.
\item\label{Generic:continuity for homoclinic class} For any hyperbolic periodic point $p$, there is a $C^1$-neighborhood $\cU$ of $f$, such that every diffeomorphism $g\in\cU\cap\cR$ is a continuity point of the map $g\mapsto H(p_g,g)$ with respect to the Hausdorff topology on the set of non-empty compact sets of $M$.
\item\label{Generic:homoclinic class coincide} Any two homoclinic classes $H(p_1,f)$ and $H(p_2,f)$ either coincide or are disjoint. Moreover, any two periodic orbits of the same index contained in a homoclinic class of $f$ are homoclinically related.
\item\label{Generic:interval of index} For any homoclinic class $H(p)$, there is an interval of natural numbers $[\alpha,\beta]$ and a $C^1$-neighborhood $\cV$ of $f$, such that, for any $g\in\cV$, the set of indices of hyperbolic periodic points contained in $H(p_g,g)$ is exactly the interval $[\alpha,\beta]$.
\item\label{Generic:cycle implies different index} Consider a periodic point $p$. If for any $C^1$-neighborhood $\cU$ of $f$, there is a diffeomorphism $g\in\cU$ having a heterodimensional cycle associated with $p_g$ and some periodic point $q$, then the homoclinic class $H(p)$ contains a periodic point whose index equals $\ind (q)$.
\item\label{Generic:preperiodic set} Consider an invariant compact $K$ which is a Hausdorff limit set of a sequence of periodic orbits $O_n$. If for any neighborhood $\cU$ of $f$ and any $N>0$, there is $g\in\cU$ and $n>N$, such that $O_n$ is a periodic orbit of index $i$, then $K$ is the Hausdorff limit set of a sequence of periodic orbits of index $i$.
\item\label{Generic:dense set of orbit with close exponents} Consider a non-trivial homoclinic class $H(p)$ of $f$. For any $\varepsilon>0$, the set
   \begin{center}
     \{$q\in Per(f)$: q\text{ has simple spectrum, and } $|\chi_i(q)-\chi_i(p)|<\varepsilon$, \text{ for } $\forall$ \text{i=1,2,$\cdots$,d}\}
   \end{center}
   is dense in $H(p)$.
\end{enumerate}
\end{theorem}

\subsection{Previous results}

We state in this subsection some previous results, some of which have already been mentioned in Section~\ref{Section:introduction}. First, we have the following result which is a combination of~\cite{dg,bdg}.

\begin{theorem}[\cite{bdg,dg}]\label{Thm:different-index}
For generic $f\in\diff^1(M)$, assume that $p$ is a hyperbolic periodic point of index $i$. If the homoclinic class $H(p)$ contains a hyperbolic point $q$ of index smaller than $i$, then there is an ergodic measure supported on $H(p)$ whose $i^{th}$ Lyapunov exponent is $0$. Moreover, if $\ind(q)=i-1$ and there is a dominated splitting $T_{H(p)}M=E\oplus F\oplus G$ such that $\dim(E)=i-1$ and $\dim(F)=1$, then, there is an ergodic measure $\mu$ such that $\rm {supp}(\mu)=H(p)$ and the $i^{th}$ Lyapunov exponent of $\mu$ is $0$.
\end{theorem}

The next result characterizes the non-hyperbolicity of a homoclinic class by the existence of weak periodic orbits contained in it for $C^1$-generic systems.

\begin{theorem}[\cite{w}]\label{Thm:weak periodic orbits}
For generic $f\in\diff^1(M)$, if a homoclinic class $H(p)$ is not hyperbolic, then there is a sequence of periodic orbits contained in $H(p)$ with a Lyapunov exponent converging to 0.
Moreover, we have the following facts.
\item[$(1)$] Assume $H(p)$ has a dominated splitting $T_{H(p)}M=E\oplus F$, where $\dim(E)=\ind (p)=i$ such that the bundle $E$ is not contracted. Then for any $\varepsilon>0$, there is a periodic orbit $q$ homoclinically related to $p$ such that $\chi_i(q)\in (-\varepsilon,0)$.
\item[$(2)$] Assume $H(p)$ has a dominated splitting $T_{H(p)}M=E\oplus F$, such that $\dim(E)$ is smaller than the smallest index of periodic orbits contained in $H(p)$, then the bundle $E$ is contracted. Symmetrically, if $\dim(E)$ is larger than the largest index of periodic orbits contained in $H(p)$, then the bundle $F$ is expanded.
\end{theorem}

Recall that for a hyperbolic invariant measure $\mu$, there is a full $\mu$-measure set $\Lambda$, such that, there is a splitting $T_{\Lambda}M=\cE^s\oplus \cE^u$ such that, the bundles $\cE^s$ and $\cE^u$ are associated to the negative and positive Lyapunov exponents respectively. We call this splitting the \emph{non-uniform hyperbolic splitting} of $\mu$ and the dimension of $\cE^s$ is called the \emph{index} of $\mu$.
The following result claims the support of a hyperbolic measure intersects a homoclinic class if the above splitting is a dominated splitting.

\begin{proposition}[Proposition 1.4 of~\cite{c}]\label{Pro:dominated-measure}
Let $\mu$ be a hyperbolic measure of index $i$. If the non-uniform hyperbolic splitting of $\mu$ is a dominated splitting, then there is a hyperbolic periodic point $p$ of index $i$, such that ${\rm supp}(\mu)\cap H(p)\neq\emptyset$. Moreover, if $\mu$ is ergodic, then ${\rm supp}(\mu)\subset H(p)$.
\end{proposition}

%% file: dominated.tex
\section{The non-hyperbolic behavior in the dominated case}\label{Section:domination}
In this section we complete the proof of Theorem~\ref{Thm:dominated}.
The following result is a combination of Proposition 7.1 and Proposition 8.1 of~\cite{bcdg}.

\begin{proposition}\label{Prop:no domination+weak orbits implies cycle}
Consider a diffeomorphism $f\in\diff^1(M)$, and a hyperbolic periodic point $p$ of index $i$, where $2\leq i\leq d-1$. Assume the following properties are satisfied:
\begin{itemize}

\item[--] $H(p)$ has no dominated splitting of index $i-1$,

\item[--] for any $\varepsilon>0$, there is a periodic point $p_{\varepsilon}$ homoclinically related to the orbit of $p$, such that $\chi_i(p_{\varepsilon})\in (-\varepsilon,0)$.

\end{itemize}

Then for any $C^1$-neighborhood $\mathcal{U}$ of $f$, there is a diffeomorphism $g\in\mathcal{U}$ having a heterodimensional cycle associated with $orb(p_g)$ and a periodic point $q_g$ of $g$ of index $i-1$.
\end{proposition}

\begin{lemma}\label{Lem:codimension one domination}
For any diffeomorphism $f\in\diff^1(M)$, consider a non-trivial homoclinic class $H(p)$ having a dominated splitting $T_{H(p)}M=E\oplus F$ such that $\dim(E)=\ind (p)=1$. If the bundle $E$ is not contracted, then there is an ergodic measure $\mu$ supported on $H(p)$, whose first Lyapunov exponent vanishes.
\end{lemma}

\begin{proof}
Since the bundle $E$ is not contracted, by Claim 1.7 of~\cite{c}, there is an ergodic measure $\mu$ supported on $H(p)$, such that $\chi_1(\mu)\geq 0$, where $\chi_1(\mu)$ is the Lyapunov exponent of $\mu$ along the bundle $E$.

If $\chi_1(\mu)=0$, then the ergodic measure $\mu$ which is supported on $H(p)$ is non-hyperbolic.

If $\chi_1(\mu)>0$, then all Lyapunov exponents of $\mu$ are positive by the dominated splitting $E\oplus F$. Then $\mu$ is supported on a periodic source, which contradicts the fact that $\rm {supp}(\mu)\subset H(p)$ and that $H(p)$ is non-trivial. This finishes the proof of Lemma~\ref{Lem:codimension one domination}.
\end{proof}

Now will manage to prove Theorem~\ref{Thm:dominated}.

\begin{proof}[Proof of Theorem~\ref{Thm:dominated}]

Now we consider a generic diffeomorphism $f\in\diff^1(M)$ which satisfies Theorem~\ref{Thm:generic properties}, a hyperbolic periodic point $p$ of index $i$, where $1\leq i\leq d-1$, and a dominated splitting $T_{H(p)}M=E\oplus F$ where $\dim(E)=i$ and $E$ is not contracted. By Theorem~\ref{Thm:different-index}, we can assume that all periodic points contained in $H(p)$ have index larger than or equal to $i$, otherwise, there is an ergodic measure supported on $H(p)$ whose $i^{th}$ Lyapunov exponent is zero and there is nothing need to prove.

If $\dim(E)=1$, then the conclusion can be obtained from Lemma~\ref{Lem:codimension one domination}. Hence we can assume that $\dim(E)\geq 2$. We consider two subcases whether the bundle $E$ has a dominated splitting $E_1\oplus E_2$ with $\dim(E_2)=1$ or not. Equivalently, we distinguish whether the homoclinic class $H(p)$ has a dominated splitting of index $i-1$ or not.

\paragraph{\bf{Case 1: $H(p)$ has a dominated splitting of index $i-1$}.}

In this case, the bundle $E$ has a dominated splitting into two bundles $E=E^s\oplus E^c$ such that $\dim(E^c)=1$. By Theorem~\ref{Thm:weak periodic orbits}, the bundle $E^s$ is contracted by $Df$. The bundle $E^c$ is not contracted, since the bundle $E$ is not contracted. By Claim 1.7 of~\cite{c}, there is an ergodic measure $\mu$ supported on $H(p)$, such that $\chi_i(\mu)\geq 0$, where $\chi_i(\mu)$ is the Lyapunov exponent of $\mu$ along the bundle $E^c$.

If $\chi_i(\mu)=0$, the conclusion of Theorem~\ref{Thm:dominated} holds.

If $\chi_i(\mu)>0$, then $\mu$ is a hyperbolic measure because the bundle $E^s$ is contracted by $Df$. Moreover, the non-uniform hyperbolic splitting of $\mu$ is a dominated spitting $E^s\oplus (E^c\oplus F)$. By Proposition~\ref{Pro:dominated-measure}, there is a hyperbolic periodic point $q$ of index $i-1$, such that $\supp(\mu)\subset H(q)$. By the item~\ref{Generic:homoclinic class coincide} of Theorem~\ref{Thm:generic properties}, $q$ belongs to $H(p)$, which contradicts the assumption that all periodic points contained in $H(p)$ have index larger than or equal to $i$.

\paragraph{\bf{Case 2: $H(p)$ has no dominated splitting of index $i-1$}.}

By Theorem~\ref{Thm:weak periodic orbits}, since the bundle $E$ is not contracted, for any $\varepsilon>0$, there is a periodic point $p_{\varepsilon}$ homoclinically related to $p$, such that $\chi_{i}(p_{\varepsilon})\in (-\varepsilon,0)$. Since $H(p)$ has no dominated splitting of index $i-1$, by Proposition~\ref{Prop:no domination+weak orbits implies cycle}, there is an arbitrarily small perturbation $g$ of $f$, such that $H(p_g,g)$ has a heterodimensional cycle associated to $orb(p_g)$ and $orb(q_g)$ with $\ind(q_g)=i-1$. Then by the item~\ref{Generic:cycle implies different index} of Theorem~\ref{Thm:generic properties}, there is a periodic point $q\in H(p)$ whose index equals $i-1$, which contradicts the assumption that all periodic points contained in $H(p)$ have index larger than or equal to $i$.

The proof of Theorem~\ref{Thm:dominated} is now complete.
\end{proof}

%% file: non-dominated.tex
\section{Non-hyperbolic ergodic measures with full support: the non-domination case}\label{Section:non-domination}

In this section we prove Theorem~\ref{Thm:non-dominated}.

\subsection{Multiple almost shadowing of $orb(p)$ with a weak Lyapunov exponent}

The following proposition states that, for generic $f\in\diff^1(M)$, if a homoclinic class $H(p)$ has no dominated splitting of index $\ind(p)$, then  by a $C^1$-small perturbation, arbitrarily dense in $H(p)$, there is a periodic orbit that multiple almost shadows the orbit of $p$ and that has a Lyapunov exponent close to 0.

\begin{proposition}\label{Pro:saddle with weak exponent}
For generic $f\in\diff^1(M)$, consider a center-dissipative hyperbolic periodic saddle $p$ of index $i$ which has simple spectrum. Assume that the homoclinic class $H(p,f)$ has no dominated splitting of index $i$. Then for any $\varepsilon,\gamma>0$, for any $\varkappa\in (0,1)$, and for any $C^1$-neighborhood $\cU$ of $f$, there are a diffeomorphism $g\in\cU$, and a hyperbolic saddle $q$ of $f$, such that:
\begin{enumerate}
\item\label{Item:homoclinically related} the saddle $q$ is homoclinically related to $p$ with respect to $f$ and $g$,
\item\label{Item:multiple almost shadow} the orbit of $p$ is $(\gamma,\varkappa)$-multiple almost shadowed by the orbit of $q$,
\item\label{Item:hausdorff close} the Hausdorff distance between $orb(q)$ and $H(p,f)$ is less than $\varepsilon$,
\item\label{Item:coincide set} $g$ coincides with $f$ on $orb(q)$ and outside a small neighborhood of $orb(q)$,
\item\label{Item:simple spectrum} the saddle $q$ has simple spectrum with respect to $g$,
\item\label{Item:weak exponent} $\chi_{i+1}(q,g)\in(0,\varepsilon)$.
\end{enumerate}

\end{proposition}

\begin{proof}
We assume that $f$ satisfies the properties stated in Lemma~\ref{Lem:m.a.s.} and Theorem~\ref{Thm:generic properties}.
Without loss of generality, we can assume that $\varepsilon>0$ is small such that any diffeomorphism $h$ that is $\varepsilon$-close to $f$ in $\diff^1(M)$ is contained in $\cU$ and such that $\varepsilon<|\chi_j|$, where $j=i,i+1$. Then there are two positive integer $T$ and $\tau$ that satisfy the conclusion of Lemma~\ref{Lem:conditions for domination of given index} associated to the constant $\varepsilon$.

By the definition of dominated splitting, there is $\eta>0$, such that, any invariant compact set that is $\eta$-close to $H(p)$ in the Hausdorff distance has no $T$-dominated splitting of index $i$. Moreover, we can assume that $\eta<\varepsilon$.

By Lemma~\ref{Lem:m.a.s.}, there is a center-dissipative periodic saddle $orb(q)$ with simple spectrum homoclinically related to $p$ such that,
\begin{itemize}
\item[--] the orbit of $p$ is $(\gamma,\varkappa)$-multiple almost shadowed by the orbit of $q$,
\item[--] the Hausdorff distance between $orb(q)$ and $H(p,f)$ is less than $\eta$,
\item[--] the Lyapunov exponents of $orb(q)$ are close to those of $orb(p)$.
\end{itemize}
Hence the items~\ref{Item:multiple almost shadow},~\ref{Item:hausdorff close} are satisfied. Moreover, the periodic point $q$ can be chosen such that its period $\pi(q)$ is larger than $\tau$.
\medskip

Now we do a perturbation to get a diffeomorphism $g$ that satisfies the item~\ref{Item:homoclinically related},~\ref{Item:coincide set},~\ref{Item:simple spectrum},~\ref{Item:weak exponent}.

Consider the hyperbolic periodic orbit $orb(q,f)$. By the choice of $\eta$, one can see that $orb(q)$ has no $T$-dominated splitting of index $i$.

By Lemma~\ref{Lem:conditions for domination of given index}, for each $n=0,1,\cdots,\pi(q)-1$, there is a one-parameter family of matrices $(A_{n,t})_{t\in [0,1]}$ in $GL(d,\RR)$, such that, denote by $B_t=A_{\pi(q)-1,t}\circ \cdots\circ A_{0,t}$ for $t\in[0,1]$, the following properties are satisfied.
\begin{itemize}
\item[--] $A_{n,0}=Df(f^n(q))$,
\item[--] $\|A_{n,t}-Df(f^n(q))\|<\varepsilon$ and $\|A^{-1}_{n,t}-Df^{-1}(f^n(q))\|<\varepsilon$, for any $t\in[0,1]$,
\item[--] $\chi_{i+1}\in(0,\varepsilon)$,
\item[--] $\chi_j(B_t)=\chi_j(B_0)$, for any $t\in[0,1]$ and for any $j\in\{1,2,\cdots,d\}\setminus\{i,i+1\}$,
\item[--] $B_t$ is hyperbolic for any $t\in [0,1]$,
\end{itemize}

Take a small constant $\delta>0$, such that the local manifolds $W^s_{\delta}(orb(q),f)$ and $W^u_{\delta}(orb(q),f)$ of size $\delta$ are two embedding sub-manifolds of dimension $i$ and $d-i$ respectively. Then there are two transverse homoclinic points $z\in W^s_{\delta}(orb(q),f)\pitchfork W^u(p,f)$ and $y\in W^u_{\delta}(orb(q),f)\pitchfork W^s(p,f)$ since the periodic orbit $orb(q)$ is homoclinically related to $p$ with respect to $f$. Consider the two small compact sets $\{z\}$ and $\{y\}$ as $K^s$ and $K^u$. There is a small neighborhood $V$ of $orb(q)$, such that $\overline{V}$ is disjoint with $orb^-(z,f)$, $orb^+(y,f)$ and $orb(p)$. By Lemma~\ref{Lem:franks-gourmelon lemma}, there is a diffeomorphism $g$ that is $\varepsilon$-close to $f$ in $\diff^1(M)$, and that satisfies the following properties:
\begin{itemize}
\item[$\mathit a)$.] $g$ coincides with $f$ on the orbit of $orb(q)$ and outside $V$;
\item[$\mathit b)$.] $z\in W^s_{\delta}(orb(q),g)$ and $y\in W^u_{\delta}(orb(q),g)$;
\item[$\mathit c)$.] $Dg(g^n(q))=Dg(f^n(q))=A_{n,1}$ for all $n=0,\cdots,\pi(q)-1$.
\end{itemize}
Then we have that $z\in W^s(orb(q),g)\cap W^u(p,g)$ and $y\in W^u(orb(q),g)\cap W^s(p,g)$, and by an arbitrarily small $C^1$-perturbation if necessary, we can assume that the intersections are transverse. Hence $orb(q)$ is still homoclinically related to $p$ under $g$, which is the item~\ref{Item:homoclinically related} in Proposition~\ref{Pro:saddle with weak exponent}. The item~\ref{Item:coincide set} is automatically satisfied by the item $\mathit a$. The items~\ref{Item:simple spectrum},~\ref{Item:weak exponent} is satisfied by the item $\mathit c$ and the properties of the one-parameter families $(A_{n,t})_{t\in[0,1];n=0,1,\cdots,\pi(q)-1}$.

The proof of Proposition~\ref{Pro:saddle with weak exponent} is now complete.
\end{proof}

\subsection{Construction of sequences of weak periodic orbits}

The following proposition gives a sequence of periodic orbits that have some shadowing properties for $C^1$-generic diffeomorphisms.

\begin{proposition}\label{Pro:sequence of periodic orbits for generic systems}
For generic $f\in\diff^1(M)$, assume that $p$ is a center-dissipative hyperbolic saddle of index $i$ with simple spectrum whose homoclinic class $H(p)$ has no dominated splitting of index $i$. Then, there is a sequence of center-dissipative periodic points $(q_{n})_{n\geq 1}$ with simple spectrum, together with a sequence of positive numbers $(\gamma_n)_{n\geq 1}$, such that, for any $n\geq 1$, the followings are satisfied.
\begin{enumerate}
\item\label{sequence:homoclinically related} $q_{n}$ is homoclinically related to $p$.
\item\label{sequence:orbit dense} $orb(q_{n},f)$ is $\frac{1}{4^n}$-dense in $H(p)$.
\item\label{sequence:m.a.s.} $\gamma_{n}<\frac{1}{2}\gamma_{n-1}$ and the orbit of $q_{n-1}$ is $(\gamma_{n-1},1-\frac{1}{2^{n-1}})$-multiple almost shadowed by the orbit of $q_{n}$.
\item\label{sequence:estimation of exponent} There exists a positive integer $N_n>\pi(q_n)$, such that for any point $x$ contained in the $2\gamma_n$-neighborhood of $orb(q_{n})$, we have $L^{(N_n)}_{d-i}(x)-L^{(N_n)}_{d-i-1}(x)\in (0,\frac{1}{2^n})$.
\end{enumerate}
\end{proposition}

\begin{proof}
Since $f$ is a $C^1$-generic diffeomorphism, by the item~\ref{Generic:continuity for homoclinic class} of Theorem~\ref{Thm:generic properties}, for a diffeomorphism $g$ close to $f$ in $\diff^1(M)$, the homoclinic class $H(p_g,g)$ is close to $H(p,f)$ in the Hausdorff topology.
Hence one can see that the items~\ref{Item:homoclinically related},~\ref{Item:multiple almost shadow},~\ref{Item:hausdorff close},~\ref{Item:simple spectrum},~\ref{Item:weak exponent} in Proposition~\ref{Pro:saddle with weak exponent} are persistent under $C^1$-perturbations. Therefore by a standard Baire argument, the following statement holds. \\
\textit{For generic $f\in\diff^1(M)$, assume $p$ is a center-dissipative hyperbolic saddle of index $i$ which has simple spectrum, if the homoclinic class $H(p)$ has no dominated splitting of index $i$, then for any $\varepsilon,\gamma>0$, and any $\varkappa\in(0,1)$, there is a center-dissipative periodic saddle $q$ with simple spectrum homoclinically related to $p$, such that:
\begin{itemize}
\item[--] $orb(p)$ is $(\gamma,\varkappa)$-multiple almost shadowed by $orb(q)$,
\item[--] $orb(q)$ is $\varepsilon$-dense in $H(p)$,
\item[--] $\chi_{i+1}(q,f)\in (0,\varepsilon)$.
\end{itemize}}
\medskip

One may assume that the diffeomorphism $f$ in the statement of Proposition~\ref{Pro:sequence of periodic orbits for generic systems} satisfies the property above and the properties of Lemma~\ref{Lem:m.a.s.} and Theorem~\ref{Thm:generic properties}. Now we construct the sequence of periodic orbits. To make it complete, we take $q_0=p$ and $\gamma_0=1$.
\medskip


Assume $orb(q_n)$ and $\gamma_n$ have been taken to satisfy the properties stated in the proposition for any $n\leq k-1$. We construct $orb(q_{k})$, $\gamma_{k}$ and $N_{k}$. We have that $H(q_{k-1})=H(p)$. Consider the periodic point $q_{k-1}$, since $H(q_{k-1})$ has no domination of index $i$, and by the choice of $\cR$, there is a periodic point $q_{k}$ with simple spectrum homoclinically related to $q_{k-1}$, such that, $orb(q_{k-1})$ is $(\gamma_{k-1},1-\frac{1}{2^{k-1}})$-multiple almost shadowed by $orb(q_{k})$, $orb(q_{k})$ is $\frac{1}{4^k}$-dense in $H(p)$, and $\chi_{i+1}(q_{k},f)\in (0,\frac{1}{4^{k}})$. Hence the items~\ref{sequence:homoclinically related},~\ref{sequence:orbit dense},~\ref{sequence:m.a.s.} are satisfied. By the fact that $\chi_{i+1}(\mu,f)=\lim_{m\rightarrow+\infty} (L_{d-i}^{(m)}(x,f)-L_{d-i-1}^{(m)}(x,f))$, there is $N_{k}>\pi(q_{k})$, such that for any $x\in orb(q_{k})$, we have  $L^{(N_{k})}_{d-i}(x)-L^{(N_{k})}_{d-i-1}(x)\in (0,\frac{1}{2^{k}})$. Then there is a constant $\gamma_{k}\in (0,\frac{\gamma_{k-1}}{2})$, such that for any $x$ contained in the $2\gamma_{k}$-neighborhood of $orb(q_{k})$, we have that $L^{(N_{k})}_{d-i}(x)-L^{(N_{k})}_{d-i-1}(x)\in (0,\frac{1}{2^{k}})$. Then the item~\ref{sequence:estimation of exponent} is satisfied.

The proof of Proposition~\ref{Pro:sequence of periodic orbits for generic systems} is now complete.
\end{proof}

\subsection{Non-hyperbolic ergodic measures supported on $H(p)$: end of the proof of Theorem~\ref{Thm:non-dominated}}

Now we can prove Theorem~\ref{Thm:non-dominated}.
By the item~\ref{Generic:kupka-smale} and~\ref{Generic:dense set of orbit with close exponents} of Theorem~\ref{Thm:generic properties}, we can assume that the periodic point $p$ is center-dissipative and has simple spectrum. Then there is a sequence of center-dissipative hyperbolic periodic orbits $(q_n)$ that satisfies the properties in Proposition~\ref{Pro:sequence of periodic orbits for generic systems}.

By Lemma~\ref{Lem:conditions for ergodicity}, denoting by $\mu_n$ the probability atomic measure uniformly distributed on the orbit $orb(q_n)$ for each $n$, the weak-$*$-limit of $\mu_n$ is an ergodic measure $\mu$, whose support is:
   \begin{center}
     ${\rm supp}(\mu)=\bigcap_{n=1}^{\infty}(\overline{\bigcup_{k=n}^{\infty}orb(q_n)})=H(p)$.
   \end{center}

It only remains to show that $\mu$ is a non-hyperbolic measure, which is from the following claim.

\begin{claim}\label{Claim:non-hyperbolicity of mu}
For the ergodic measure $\mu$, we have that $\chi_{i+1}(\mu,f)=0$.
\end{claim}

\begin{proof}
Since the orbit of $q_{n}$ is a $(\gamma_n,1-\frac{1}{2^n})$-multiple almost shadowed by the orbit of $q_{n+1}$, there are a subset $\Gamma_n\subset orb(q_{n})$ and a map $\rho_n:\Gamma_n\mapsto orb(q_{n+1})$ for each $n\geq 2$ from Definition~\ref{Def:multiple almost shadow}. Take $Y_n=\rho_n^{-1}\circ\cdots\circ\rho_2^{-1}(orb(q_{1}))$, we can see that $Y_n$ is well-defined. Take the upper topological limit
   \begin{displaymath}
     Y=\limsup_{n\rightarrow+\infty} Y_n.
   \end{displaymath}
Since $Y$ is a compact set and $\mu$ is the limit measure of $\mu_n$, we have that
   \begin{displaymath}
      \mu (Y)\geq \limsup_{n\rightarrow+\infty} \mu_n(Y_n)\geq \prod_{k=1}^{n-1}(1-\frac{1}{2^{n}})>0.
   \end{displaymath}

By the fact that $\gamma_{n+1}<\frac{1}{2}\gamma_n$, we can see that the set $Y$ is contained in the $2\gamma_n$-neighborhood of $orb(q_{n})$ for every $n\geq 1$. Then for any $x\in Y$, and any $n\geq 1$, we have that $L^{(N_n)}_{d-i}(x)-L^{(N_n)}_{d-i-1}(x)\in (0,\frac{1}{2^n})$ for the strictly increasing sequence $(N_n)_{n\geq 1}$. By the facts that $\mu$ is ergodic and $\mu(Y)>0$, we have that for $\mu$-a.e. $x\in Y$,
   \begin{displaymath}
     \chi_{i+1}(\mu,f)=\lim_{m\rightarrow+\infty} (L_{d-i}^{(m)}(x,f)-L_{d-i-1}^{(m)}(x,f))=\lim_{n\rightarrow+\infty} (L_{d-i}^{(N_n)}(x,f)-L_{d-i-1}^{(N_n)}(x,f))=0.
   \end{displaymath}
\end{proof}

%% file: proof-of-consequences.tex
\section{Proof of the other propositions}\label{Section:Proof of consequences}

\begin{proof}[Proof of Proposition~\ref{Pro:non-hyperbolic measure with full support}]
By the items~\ref{Generic:homoclinic class coincide},~\ref{Generic:interval of index} of Theorem~\ref{Thm:generic properties}, we can assume that there is a periodic point $q$ such that $H(q)=H(p)$, and $p,q$ have indices $i,i+1$ respectively. We consider the following two cases.

If $H(p)$ has no dominated splitting of index $i$ or of index $i+1$, then by Theorem~\ref{Thm:non-dominated}, there is a non-hyperbolic ergodic measure $\mu$ such that $\rm {supp}(\mu)=H(p)$.

If otherwise, the homoclinic class $H(p)$ has both a dominated splitting of index $i$ and a dominated splitting of index $i+1$, then by Theorem~\ref{Thm:different-index}, there is a non-hyperbolic ergodic measure $\mu$ such that $\rm {supp}(\mu)=H(p)$.
\end{proof}
\bigskip

\begin{proof}[Proof of Proposition~\ref{Pro:partial hyperbolicity versus cycle}]
We will assume that every periodic point contained in $H(p)$ has index $i$, since otherwise the first case in the statement holds. We consider the following two possibilities.

Assume that there is a dominated splitting $T_{H(p)}M=E\oplus F$ with $\dim(E)=i-1$, then by Theorem~\ref{Thm:weak periodic orbits} and the fact that every periodic point contained in $H(p)$ has index $i$, the bundle $E$ is contracted by $Df$, hence $E\oplus F$ is a partially hyperbolic splitting $E^s\oplus E^c$. Similarly if $\dim(E)=i+1$, we have that $F$ is expanded by $Df$ and $E\oplus F$ is a partially hyperbolic splitting $E^c\oplus E^u$. The second case stated in Proposition~\ref{Pro:partial hyperbolicity versus cycle} holds.

Assume otherwise that $H(p)$ admits neither domination of index $i-1$ nor domination of index $i+1$. Since $H(p)$ is not hyperbolic, by Theorem~\ref{Thm:weak periodic orbits} and the fact that every periodic point contained in $H(p)$ has index $i$, for any $\varepsilon>0$, there is a periodic point $q\in H(p)$, such that $\chi_i(q)\in (-\varepsilon,0)$ or $\chi_{i+1}(q)\in (0,\varepsilon)$. Then by Proposition~\ref{Prop:no domination+weak orbits implies cycle}, the diffeomorphism $f$ can be $C^1$-approximated by diffeomorphisms with a heterodimensional cycle in the homoclinic class. By the items~\ref{Generic:interval of index},~\ref{Generic:cycle implies different index} of Theorem~\ref{Thm:generic properties}, there is a periodic point of different index contained in $H(p)$, which contradicts the assumption.
\end{proof}
\bigskip

\begin{proof}[Proof of Proposition~\ref{Pro:partial hyperbolicity versus cycle of index 1}]
We will assume that all periodic points contained in $H(p)$ have the same index $d-1$, otherwise the first case in the statement holds.

Assume that there is a dominated splitting $T_{H(p)}M=E\oplus F$ with $\dim(E)=d-2$, then by Theorem~\ref{Thm:weak periodic orbits} and the fact that every periodic point contained in $H(p)$ has index $d-1$, the bundle $E$ is contracted by $Df$, hence $E\oplus F$ is a partially hyperbolic splitting $E^s\oplus E^c$, which is the second case in the statement.

Assume otherwise that $H(p)$ admits no domination of index $d-2$. Since $H(p)$ is non-hyperbolic, by Theorem~\ref{Thm:weak periodic orbits} and the fact that every periodic point contained in $H(p)$ has index $d-1$, there is a sequence of periodic points $(q_n)$ contained in $H(p)$, such that $\chi_{d-1}(q_n)\rightarrow 0^-$ or $\chi_{d}(q_n)\rightarrow 0^+$ as $n\rightarrow +\infty$. Moreover, one can choose $q_n$ to converge to $H(p)$ in the Hausdorff topology. If $\chi_{d-1}(q_n)\rightarrow 0^-$ occurs, then by Proposition~\ref{Prop:no domination+weak orbits implies cycle} one can get a heterodimensional cycle in the homoclinic class by arbitrarily $C^1$-small perturbation. By the items~\ref{Generic:interval of index},~\ref{Generic:cycle implies different index} of Theorem~\ref{Thm:generic properties}, there is a periodic point of different index contained in $H(p)$, which contradicts the assumption. If otherwise $\chi_{d}(q_n)\rightarrow 0^+$, then for any $N$ and for any neighborhood $\cU$ of $f$ in $\diff^1(M)$, there is $g\in\cU$ and $n>N$, such that $q_n$ is a periodic sink of $g$. By the item~\ref{Generic:preperiodic set} of Theorem~\ref{Thm:generic properties}, there is a sequence of sinks converges to $H(p)$, which is the third case in the statement.
\end{proof}
\bigskip

\begin{proof}[Proof of Proposition~\ref{Pro:classify homoclinic classes}]
We consider a generic diffeomorphism $f\in\diff^1(M)$ which satisfies Theorem~\ref{Thm:generic properties}, and a non-hyperbolic homoclinic class $H(p)$ associated to a hyperbolic periodic point $p$ of index $i$. We assume that Case (a) does not occur, which means that all periodic points contained in $H(p)$ have the same index $i$. We consider the following possibilities.

Assume that there is a dominated splitting $T_{H(p)}M=E\oplus F$ with $\dim(E)=i$. If the bundle $E$ is not uniformly contracted, then the bundle $E$ has a dominated splitting $E^s\oplus E^c_1$ with $\dim(E^s)=i-1$. Otherwise, using Proposition~\ref{Prop:no domination+weak orbits implies cycle} and Theorem~\ref{Thm:weak periodic orbits}, one can get a heterodimensional cycle by arbitrarily $C^1$-small perturbation and by the items~\ref{Generic:interval of index},~\ref{Generic:cycle implies different index} of Theorem~\ref{Thm:generic properties}, there is a periodic point of different index contained in $H(p)$, which contradicts the assumption. Moreover, by Theorem~\ref{Thm:weak periodic orbits}, the bundle $E^s$ is uniformly contracted. Symmetrically, if the bundle $F$ is not uniformly expanded, then it can be split as $E^c_2\oplus E^u$, where $\dim(E^c_2)=1$ and $E^u$ is uniformly expanded. Hence in this case, the homoclinic class has a partially hyperbolic splitting $T_{H(p)}M=E^s\oplus E^c_1\oplus E^c_2\oplus E^u$, with $\dim(E^c_j)\in \{0,1\}$, and $i=\dim(E^s\oplus E^c_1)$. This is Case (b).

Assume now that the homoclinic class $H(p)$ admits no domination of index $i$. Consider the finest dominated splitting over $H(p)$. Combine all the contracted (center and expanded resp.) bundles and denote it by $E^s\oplus E^c\oplus E^u$, such that $E^s$ and $E^u$ are uniformly contracted and expanded bundles respectively and $E^c$ is the center bundle. By Proposition~\ref{Pro:partial hyperbolicity versus cycle} and Proposition~\ref{Pro:partial hyperbolicity versus cycle of index 1}, we have that $i=dim(E^s)+1$ or $i=dim(E^s\oplus E^c)-1$. We then consider the following two subcases.

If there exist both periodic points which contract and others which expand the volume along $E^c$, then there are both periodic orbits whose $i^{th}$ exponent arbitrarily close to 0, and those whose $({i+1})^{th}$ exponent arbitrarily close to 0. Then the homoclinic class $H(p)$ admits both a domination of index $i-1$ and a domination of index $i+1$. Otherwise, using Proposition~\ref{Prop:no domination+weak orbits implies cycle}, one can get a heterodimensional cycle by arbitrarily $C^1$-small perturbation, and get a periodic point of different index in $H(p)$. Then we have that $T_H{H(p)}=E^s\oplus E^c\oplus E^u$, with $\dim(E^s)=i-1$ and $\dim(E^c)=2$, and there is no finer dominated splitting along $E^c$. This is Case (c).

If there exist only periodic points which contract the volume along $E^c$, then arguing as in the previous case, the homoclinic class $H(p)$ admits a domination $E\oplus E^u$ where $\dim(E)=i+1$. This is Case (d). Symmetrically, if there exist only periodic points which expand the volume along $E^c$, then it is Case(d').
\end{proof}